\documentclass[reqno]{amsart}
\usepackage{amsmath, amssymb, mathtools, graphics}
\usepackage{tikz}
\usetikzlibrary{arrows}
\usetikzlibrary{matrix}

\usepackage[utf8]{inputenc}
\usepackage[T1]{fontenc}

\DeclareMathAlphabet{\mathpzc}{OT1}{pzc}{m}{it}

\newcommand{\mat}[1]{\left[\begin{smallmatrix}#1\end{smallmatrix}\right]}
\newcommand{\Mat}[1]{{\smaller\left[\;\begin{matrix}#1\end{matrix}\;\right]}}

\newcommand{\pfa}{\eta}
\newcommand{\C}{\mathcal{C}}

\newcommand{\ZZ}{\mathbb{Z}}           
\newcommand{\sym}{\mathfrak{S}}        
\newcommand{\I}{\mathcal{I}}           
\newcommand{\Poset}{\mathfrak{P}}      
\newcommand{\UP}{\mathcal{P}}          
\newcommand{\Mad}{\mathfrak{A}}        
\newcommand{\ourmap}{\Lambda}          
\newcommand{\jumpmap}{f}               
\newcommand{\duckmap}{g}               
\newcommand{\robbieclass}{\mathcal{M}} 
\newcommand{\setsystem}{E}             
\newcommand{\Ascent}{\mathcal{A}}      
\newcommand{\madmap}{\phi}             
\newcommand{\pirate}{\psi}             

\renewcommand{\Alph}{\mathrm{Alph}}    
\DeclareMathOperator{\Par}{Par}        
\DeclareMathOperator{\Comp}{Comp}      
\DeclareMathOperator{\Mono}{Mono}      
\DeclareMathOperator{\mono}{mono}
\DeclareMathOperator{\BiPar}{BiPar}    
\DeclareMathOperator{\DiComp}{DiComp}  
\DeclareMathOperator{\BiComp}{BiComp}  
\DeclareMathOperator{\RLE}{RLE}        
\DeclareMathOperator{\RLR}{RLR}        
\DeclareMathOperator{\nz}{nz}          
\DeclareMathOperator{\col}{col}
\DeclareMathOperator{\row}{row}

\DeclareMathOperator{\cdes}{cdes}      
\DeclareMathOperator{\sdes}{sdes}      
  
\DeclareMathOperator{\asc}{asc}
\DeclareMathOperator{\Min}{Min}
\DeclareMathOperator{\Seq}{Seq}
\DeclareMathOperator{\Card}{Card}
\DeclareMathOperator{\Cyc}{Cyc}
\DeclareMathOperator{\LMin}{LMin}

\def\e{
\begin{tikzpicture}[baseline=-2pt]
  \tikzstyle{1} = [fill=black]
  \tikzstyle{0} = [fill=white]
    \matrix[nodes={draw=black,thin}, 
      ampersand replacement=\&, 
      row sep=0.4pt,column sep=0.4pt, inner sep=2pt] {
    \node[1] {}; \\
    };
\end{tikzpicture}
}

\def\t#1,#2,#3,{
\begin{tikzpicture}[baseline=-2pt]
  \tikzstyle{1} = [fill=black]
  \tikzstyle{0} = [fill=white]
    \matrix[nodes={draw=black,thin}, 
      ampersand replacement=\&, 
      row sep=0.4pt,column sep=0.4pt, inner sep=2pt] {
    \node[#1] {}; \&
    \node[#2] {}; \\ \& 
    \node[#3] {}; \\
    };
\end{tikzpicture}
}

\def\p#1,#2,#3,#4,{
\begin{tikzpicture}[baseline=-2.2pt]
  \tikzstyle{1} = [fill=black]
  \tikzstyle{0} = [fill=white]
    \matrix[nodes={draw=black,thin}, 
      ampersand replacement=\&,
      row sep=0.4pt,column sep=0.4pt, inner sep=2pt] {
    \node[#1] {}; \\
    \node[#2] {}; \&
    \node[#3] {}; \\ \& 
    \node[#4] {}; \\
    };
\end{tikzpicture}
}

\def\one{\ensuremath{\Mat{\!1\!}}}

\def\two#1,#2,{
  \begin{tikzpicture}[baseline=-2pt]
    \matrix[nodes={draw=black,thin},
      ampersand replacement=\&, 
      row sep=-0.5pt,column sep=0pt, inner sep=2pt] {
      \node[font=\scriptsize] {#1}; \\
      \node[font=\scriptsize] {#2}; \\
    };
  \end{tikzpicture}
}

\newcommand{\V}[3]{
  \begin{tikzpicture}[xscale=0.3, yscale=0.45, baseline=23pt]
    \tikzstyle{every node} = [font=\footnotesize];
    \tikzstyle{disc} = [ 
      circle,fill=black,draw=black, 
      minimum size=3.3pt, inner sep=0pt];
    \path
    node [disc] (x) at (1, 2) {}
    node [disc] (y) at (3, 2) {}
    node [disc] (z) at (2, 1) {};
    \draw 
    (x) node[above] {$#1$} -- (z) node[below] {$#3$}
    (y) node[above] {$#2$} -- (z)
    ;
  \end{tikzpicture}
}

\newtheorem{theorem}{Theorem}
\newtheorem{lemma}[theorem]{Lemma}
\newtheorem{proposition}[theorem]{Proposition}
\newtheorem{corollary}[theorem]{Corollary}

\theoremstyle{definition}
\newtheorem{definition}[theorem]{Definition}

\newtheorem{example}[theorem]{Example}

\title{Partition and Composition matrices}
\author[A. Claesson]{Anders Claesson}
\author[M. Dukes]{Mark Dukes}
\author[M. Kubitzke]{Martina Kubitzke}
\address{A. Claesson and M. Dukes: Department of Computer and Information Sciences, University of Strathclyde, Glasgow, G1 1XH, UK.}
\address{M. Kubitzke: Fakult\"at f\"ur Mathematik, Universit\"at Wien, Nordbergstra{\ss}e 15, A-1090 Vienna, Austria.}
\thanks{All authors were supported by grant No. 090038012 from the Icelandic Research Fund.}
\date{\today}
\begin{document}
\maketitle
\thispagestyle{empty}

\begin{abstract}
  This paper introduces two matrix analogues for set partitions. A
  composition matrix on a finite set $X$ is an upper triangular matrix whose entries
  partition $X$, and for which there are no rows
  or columns containing only empty sets. A partition matrix is a
  composition matrix in which an order is placed on where entries may
  appear relative to one-another.
  
  We show that partition matrices are in one-to-one correspondence
  with inversion tables. Non-decreasing inversion tables are shown to
  correspond to partition matrices with a row ordering
  relation. Partition matrices which are $s$-diagonal are classified in
  terms of inversion tables. Bidiagonal partition matrices are
  enumerated using the transfer-matrix method and are equinumerous
  with permutations which are sortable by two pop-stacks in parallel.

  We show that composition matrices on $X$ are in one-to-one correspondence
  with $(2+2)$-free posets on $X$. 
  Also, composition matrices whose rows satisfy a column-ordering relation
  are shown to be in one-to-one correspondence with parking functions. 
  Finally, we show
  that pairs of ascent sequences and permutations are in
  one-to-one correspondence with $(2+2)$-free posets whose elements are
  the cycles of a permutation, and use this relation to give an
  expression for the number of $(2+2)$-free posets on $\{1,\dots,n\}$.
\end{abstract}

\section{Introduction}

We present two matrix analogues for set partitions that are
intimately related to both permutations and $(2+2)$-free posets.

\begin{example}\label{example:A}
  Here is an instance of what we shall call a partition matrix:
  $$A = \Mat{ 
    \{1,2,3\} & \emptyset & \{5,7,8\} & \{9\} \\
    \emptyset & \{4\}     & \{6\}     & \{11\} \\
    \emptyset & \emptyset & \emptyset & \{13\} \\
    \emptyset & \emptyset & \emptyset & \{10,12\} 
  }.
  $$ 
\end{example}

\begin{definition}\label{def:M}
  Let $X$ be a finite subset of $\{1,2,\dots\}$.
  A \emph{partition matrix} on $X$ is an upper
  triangular matrix over the powerset of $X$ satisfying the
  following properties:
  \begin{enumerate}
  \item[(i)] each column and row contain at least one non-empty set;
  \item[(ii)] the non-empty sets partition $X$;
  \item[(iii)] $\col(i)<\col(j) \implies i<j$,  
  \end{enumerate}
  where $\col(i)$ denotes the column in which $i$ is a member. Let
  $\Par_n$ be the collection of all partition matrices on $[1,n]=\{1,\dots,n\}$.
\end{definition}

For instance,
\begin{align*}
  \Par_1 &= \{\mat{\{1\}}\};\\
  \Par_2 &= \left\{ \mat{\{1,2\}},\mat{\{1\}&\emptyset\\ \emptyset&\{2\}} \right\};\\
  \Par_3 &= \left\{ \mat{\{1,2,3\}},
  \mat{ \{1,2\} & \emptyset \\ \emptyset& \{3\} },
  \mat{ \{1\} & \{2\} \\ \emptyset & \{3\}},
  \mat{ \{1\} & \{3\} \\ \emptyset & \{2\}},
  \mat{ \{1\} & \emptyset \\ \emptyset &\{2,3\}},
  \mat{ 
    \{1\} & \emptyset & \emptyset\\ 
    \emptyset & \{2\} & \emptyset \\ \emptyset & \emptyset &\{3\}
  }
  \right\}.
\end{align*}
In Section \ref{pmit} we present a bijection between $\Par_n$ and
the set of \emph{inversion tables} 
$$\I_n = [0,0] \times [0,1] \times \dots \times [0,n-1],\,\text{ where }\,
[a,b]=\{i\in\ZZ: a\le i\le b\}.
$$ 
  
Non-decreasing inversion tables are shown to correspond to partition
matrices with a row ordering relation. Partition matrices which are
$s$-diagonal are classified in terms of inversion tables. Bidiagonal
partition matrices are enumerated using the transfer-matrix method and
are equinumerous with permutations which are sortable by two
pop-stacks in parallel.

In Section \ref{lpsec} we show that composition matrices on $X$ are in
one-to-one correspondence with $(2+2)$-free posets on $X$. 
We also show that composition matrices whose rows satisfy a 
column-ordering relation are in one-to-one correspondence 
with parking functions. 

Finally, in Section \ref{pcp} we show that pairs of
ascent sequences and permutations are in one-to-one correspondence
with $(2+2)$-free posets whose elements are the cycles of a
permutation, and use this relation to give an expression for the
number of $(2+2)$-free posets on $[1,n]$.

Taking the entry-wise cardinality of the matrices in $\Par_n$ one
gets the matrices of Dukes and Parviainen ~\cite{dp}. In that sense,
we generalize the paper of Dukes and Parviainen in a similar way as
Claesson and Linusson \cite{cl} generalized the paper of
Bousquet-M\'elou et al.\ ~\cite{bcdk}.  We note, however, that if we
restrict our attention to those inversion tables that enjoy the
property of being an {\it{ascent sequence}}, then we do \emph{not}
recover the bijection of Dukes and Parviainen.

\section{Partition matrices and inversion tables}\label{pmit}

For $w$ a sequence let $\Alph(w)$ denote the set of distinct entries in
$w$. In other words, if we think of $w$ as a word, then $\Alph(w)$ is
the (smallest) alphabet on which $w$ is written. Also, let us write
$\{a_1,\dots,a_k\}_<$ for a set whose elements are listed in
increasing order, $a_1<\dots <a_k$. Given an inversion table
$w=(x_1,\dots,x_n)\in\I_n$ with $\Alph(w)=\{y_1,\dots,y_k\}_<$ define
the $k\times k$ matrix $A=\ourmap(w)\in\Par_n$ by
$$A_{ij} = \big\{\, \ell :\, x_\ell=y_i \text{ and } y_j<\ell\le y_{j+1}\,\big\},
$$ 
where we let $y_{k+1}=n$. For example, with 
\begin{align*}
  w &= (0,0,0,3,0,3,0,0,0,8,3,8)\in\I_{12}
\intertext{we have $\Alph(w)=\{0,3,8\}$ and}
\ourmap(w) &= \Mat{
  \{1,2,3\} & \{5,7,8\} & \{9\}     \\ 
  \emptyset & \{4,6\}   & \{11\}    \\ 
  \emptyset & \emptyset & \{10,12\} 
}\in\Par_{12}.
\end{align*}

We now define a map $K:\Par_n\to\I_n$. Given $A \in \Par_n$, for
$\ell\in [1,n]$ let $x_\ell = \min(A_{\ast i})-1$ where $i$ is the row
containing $\ell$ and $\min(A_{\ast i})$ is the smallest entry in
column $i$ of $A$. Define
$$K(A)=(x_1,\dots,x_n).
$$

\begin{theorem}\label{thm:main}
  The map $\ourmap: \I_n \to \Par_n$ is a bijection and $K$ is
  its inverse.
\end{theorem}

\begin{proof}
  It suffices to show the following four statements:
  \begin{enumerate}
  \item $\ourmap(\I_n)\subseteq\Par_n$;            \label{p1}
  \item $K(\Par_n)\subseteq\I_n$;                  \label{p2}
  \item $K(\ourmap(w))=w$ for all $w$ in $\I_n$;   \label{p3}
  \item $\ourmap(K(A))=A$ for all $A$ in $\Par_n$. \label{p4}
  \end{enumerate}\smallskip

  Proof of~\eqref{p1}: Assume that $w=(x_1,\dots,x_n)\in\I_n$ with
  $\Alph(w)=\{y_1,\dots,y_k\}_<$, and let $A=\ourmap(w)$. We first need to
  see that $A$ is upper triangular. Let $i>j$ and consider the entry $A_{ij}$.
  Assume that $x_{\ell}=y_i$. Since $w\in\I_n$ we have $\ell>x_{\ell}$ and thus $\ell>y_i$.
  Since $y_1 < \dots <y_k$ and $i\geq j+1$ we have
  $\ell>y_i\geq y_{j+1}$. Thus $A_{ij}=\emptyset$ if $i>j$; that is, $A$ is
  upper triangular.

  Denote by
  $A_{i\ast}$ and $A_{\ast j}$ the union of the sets in the $i$th row
  and the $j$th column of $A$, respectively. By definition, we have
  $A_{i\ast}=\{\ell: x_\ell = y_i \}$ and $A_{\ast j}=[y_j+1,y_{j+1}]$
  and clearly both sets are non-empty. Thus $A$ satisfies condition
  (i) of Definition~\ref{def:M}. To show (ii), it suffices to note
  that the entries $A_{i \ast}$ form a partition of $[1,n]$, and
  so do the entries $A_{\ast j}$. To show (iii), let $u,v\in [1,n]$ with
  $\col(u)<\col(v)$. Also, let $p=\col(u)$ and $q=\col(v)$. Then $u\le
  y_{p+1}$ and $y_q<v$. Since $p+1\leq q$ and the numbers $y_i$ are
  increasing, it follows that $u\le y_{p+1}\le y_q<v$.\medskip

  Proof of~\eqref{p2}: Given $A\in\Par_n$ choose any $\ell\in
  [1,n]$. Suppose that $\ell$ is in row $i$ of $A$ and let $a =
  \min(A_{\ast i})$ be the smallest entry in column $i$ of $A$. If
  $\col(a)=\col(\ell)$ then $a\le\ell$, and so $x_\ell=a-1 \leq
  \ell-1$. Otherwise, $\col(a)<\col(\ell)$ and so, from condition
  (iii) of Definition~\ref{def:M}, we have $a<\ell$. Thus
  $x_\ell<\ell-1$.\medskip

  Proof of~\eqref{p3}: Let $w=(x_1,\dots,x_n)\in\I_n$,
  $\Alph(w)=\{y_1,\dots,y_k\}_<$, $A=\ourmap(w)$ and
  $K(A)=(z_1,\dots,z_m)$. From the definitions of $\ourmap$ and $K$ it
  is clear that $n=m$. Suppose that $\ell\in [1,n]$ is in row $i$ of
  $A$; then $x_\ell=y_i$. Also, by the definition of
  $\ourmap$, the smallest entry in column $i$ of $A$ is $y_i+1$. From
  the definition of $K$ we have $z_\ell = (y_i+1) -1 = y_i = x_\ell$.
  So $x_\ell=z_\ell$ for all $\ell\in[1,n]$, and hence $w=z=K(A)$.\medskip

  Proof of~\eqref{p4}: Let $A\in\Par_n$, $K(A)=w=(x_1,\dots,x_n)$,
  $\Alph(w)=\{y_1,\dots,y_k\}_<$ and $P=\ourmap(w)$. Also, define
  $z_j=\min(A_{\ast j})-1$. Then, for $\ell\in[1,n]$, we have
  \begin{equation}\label{eq:p4}
    \ell\in A_{ij} \iff x_\ell = z_i \text{ and } \ell\in[z_j+1, z_{j+1}]
  \end{equation}
  by the definitions of $K$ and $z_j$. In particular, this means that
  each $x_\ell$ equals some $z_i$ and, similarly, each $z_i$ equals
  some $x_\ell$. Hence $\Alph(w)=\{z_1,\dots,z_{\dim(A)}\}_<$ and it
  follows that $\dim(A)=k$ and $y_j=z_j$ for all $j\in [1,k]$. So we
  can restate \eqref{eq:p4} as
  $$\ell\in A_{ij} \iff x_\ell = y_i \text{ and } \ell\in[y_j+1, y_{j+1}].
  $$ By the definition of $\ourmap$, the right-hand side is equivalent
  to $\ell\in P_{ij}$.  Thus $A=P$.
\end{proof}

\subsection{Statistics on partition matrices and inversion tables}

Given $A \in \Par_n$, let $\Min(A)=\{\min(A_{\ast j}) :
j\in[1,\dim(A)]\}$.  For instance, the matrix $A$ in
Example~\ref{example:A} has $\Min(A)=\{1,4,5,9\}$.  From the
definition of $\ourmap$ the following proposition is apparent.

\begin{proposition}\label{prop:alph}
  If $w\in \I_n$, $\Alph(w)=\{y_1,\dots,y_k\}_{<}$ and $A=\ourmap(w)$,
  then
  $$
  \Min(A) = \{y_1+1,\dots,y_k+1\}\text{ and }
  \dim(A)=|\Alph(w)|.
  $$
\end{proposition}

\begin{corollary}
  The statistic $\dim$ on $\Par_n$ is Eulerian.
\end{corollary}
\begin{proof}
  The statistic ``number of distinct entries in the inversion table''
  is Eulerian; that is, it has the same distribution on $\sym_n$ (the
  symmetric group) as the number of descents. For a proof due to
  Deutsch see ~\cite[Corollary 19]{cl}.
\end{proof}

Let us say that $i$ is a {\it special descent} of $w=(x_1,\dots ,x_n)
\in \I_n$ if $x_i>x_{i+1}$ and $i$ does not occur in $w$. Let
$\sdes(w)$ denote the number of special descents of $w$, so
$$\sdes(w) = 
|\{\, i : x_i > x_{i+1}\,\text{ and } x_{\ell} \neq i\text{ for all }\ell\in [1,n]\,\}|.
$$
Claesson and Linusson~\cite{cl} conjectured that $\sdes$
has the same distribution on $\I_n$ as the so-called bivincular
pattern $p=(231,\{1\},\{1\})$ has on $\sym_n$. An occurrence of $p$ in
a permutation $\pi=a_1\dots a_n$ is a subword $a_ia_{i+1}a_j$ such
that $a_{i+1}>a_i=a_j+1$. We shall define a statistic on partition
matrices that is equidistributed with $\sdes$. Given $A\in\Par_n$ let
us say that $i$ is a \emph{column descent} if $i+1$ is in the same
column as, and above, $i$ in $A$. Let $\cdes(A)$ denote the number of
column descents in $A$, so
$$\cdes(A) = 
|\{\, i : \row(i)>\row(i+1)\,\text{ and }\col(i)=\col(i+1) \,\}|.
$$

\begin{proposition}
  The special descents of $w\in\I_n$ equal the column descents of $\ourmap(w)$.
\end{proposition}

\begin{proof}
  Given $t\in [1,n-1]$, let
  $u=\row(t+1)$ and $v=\row(t)$. As before, let $w=(x_1,\dots,x_n)$ and
  $\Alph(w)=\{y_1,\dots,y_k\}_{<}$. By
  the definition of $\ourmap$ we have $x_{t+1}=y_u$ and $x_t=y_v$. So,
  since the numbers $y_i$ are increasing, we have
  $$\row(t+1)<\row(t) \iff u < v \iff y_u < y_v \iff x_{t+1} < x_t.
  $$
  Now, let $u=\col(t+1)$ and $v=\col(t)$. Then
  \begin{align*}
    \col(t+1) = \col(t) 
    &\iff t+1\in [y_u+1, y_{u+1}]\text{ and } t\in[y_v+1,y_{v+1}] \\
    &\iff y_u+1 \leq t\leq y_{v+1}-1 \\
    &\iff x_j \neq t\text{ for all } j\in [1,n],
  \end{align*}
  which concludes the proof.
\end{proof}

\begin{corollary}
  The statistic $\sdes$ on $\I_n$ has the same distribution as
  $\cdes$ on $\Par_n$.
\end{corollary}

\subsection{Non-decreasing inversion tables and partition matrices}

Let us write $\Mono_n$ for the collection of matrices in $\Par_n$ which
satisfy
\begin{enumerate}
\item[(iv)] $\row(i)<\row(j) \implies i<j$,
\end{enumerate}
where $\row(i)$ denotes the row in which $i$ is a member. We say that
an inversion table $(x_1,\dots , x_n)$ is {\it non-decreasing} if
$x_i\leq x_{i+1}$ for all $1\leq i <n$.

\begin{theorem}
   Under the map $\ourmap: \Par_n \to \I_n$, matrices in $\Mono_n$
   correspond to non-decreasing inversion tables.
\end{theorem}

\begin{proof}
  Let $w=(x_1,\dots,x_n)\in\I_n$ and $\Alph(w)=\{y_1,\dots,y_k\}_{<}$.
  Looking at the definition of
  $\ourmap$, we see that $\ell$ is in row $i$ of $\ourmap(w)$ if and
  only if $x_\ell=y_i$. Since the numbers $y_i$ are increasing it follows
  that condition (iv) holds if and only if $w$ is non-decreasing, as
  claimed.
\end{proof}

\begin{proposition}\label{prop:mono}
  $|\Mono_n|={2n \choose n}/(n+1)$, the $n$th Catalan number.
\end{proposition}

\begin{proof}
  Consider the set of lattice paths in the plane from $(0,0)$ to
  $(n,n)$ which take steps in the set $\{(1,0),(0,1)\}$ and never go
  above the diagonal line $y=x$. Such paths are commonly known as {\it
    Dyck paths}, and they can be encoded as a sequence $(x_1,\dots ,
  x_n)$ where $x_i$ is the $y$-coordinate of the $i$th horizontal step
  $(1,0)$.  The restriction on such a sequence, for it to be Dyck
  path, is precisely that it is a non-decreasing inversion table.  The
  number of Dyck paths from $(0,0)$ to $(n,n)$ is given by the $n$th
  Catalan number.
\end{proof}

We want to remark that a matrix $A\in\Mono_n$ is completely determined
by the cardinalities of its entries. Thus, we can identify $A$ with an
upper triangular matrix that contains non-negative entries which sum
to $n$ and such that there is at least one non-zero entry in each row
and column. Also, the number of such matrices with $k$ rows is equal
to the Narayana number
$$\frac{1}{k}\binom{n}{k}\binom{n}{k-1}.
$$

\subsection{$s$-diagonal partition matrices}

A $k\times k$ matrix $A$ is called \emph{$s$-diagonal} if $A$ is upper
triangular and $A_{ij}=\emptyset$ for $j-i\geq s$. For $s=1$ we get
the collection of diagonal matrices.

\begin{theorem}
  Let $w=(x_1,\dots,x_n)\in \I_n$, $A=\ourmap(w)$ and
  $\Alph(w)=\{y_1,\dots, y_k\}_{<}$.  Define $y_{k+1}=n$. The matrix
  $A$ is $s$-diagonal if and only if for every $i \in [1,n]$ there
  exists an $a(i)\in [1,k]$ such that
  $$y_{a(i)}< i \leq y_{a(i)+1} \text{ and }
  x_i \in \{y_{a(i)},y_{a(i)-1},\dots , y_{\max(1,a(i)-s+1)}\}.$$
\end{theorem}

\begin{proof}
  From the definition of $\ourmap$ we have that $i$ is in row $a-j$ of
  $A$ precisely when $x_i=y_{a-j}$, and that $i$ is in column $a$ of
  $A$ precisely when $y_a < i \leq y_{a+1}$.  The matrix $A$ is
  $s$-diagonal if for every entry $i$, there exists $a(i)$ such that
  $i$ is in column $a(i)$ and row $a(i)-j$ for some $0\leq j<s$.
\end{proof}

Setting $s=1$ in the above theorem gives us the class of diagonal
matrices.  These admit a more explicit description which we will now
present.  

In computer science, \emph{run-length encoding} is a simple form of
data compression in which consecutive data elements (runs) are stored
as a single data element and its multiplicity. We shall apply this to inversion
tables, but for convenience rather than compression purposes. Let
$\RLE(w)$ denote the run-length encoding of the inversion table $w$.
For example, 
$$\RLE(0,0,0,0,1,1,0,2,3,3)=(0,4)(1,2)(0,1)(2,1)(3,2).
$$

A sequence of positive integers $(u_1,\dots ,u_k)$ which sum
to $n$ is called an \emph{integer composition} of $n$ and we write this as
$(u_1,\dots ,u_k)\models n$.

\begin{corollary}
  The set of diagonal matrices in $\Par_n$ is the image under $\ourmap$ of
  $$\{\, w\in \I_n : 
  (u_1,\dots ,u_k) \models n \text{ and } 
  \RLE(w) = (p_0,u_1)\dots (p_{k-1},u_k)\,\},
  $$ 
  where $p_0=0$, $p_1=u_1$, $p_2=u_1+u_2$, $p_3=u_1+u_2+u_3$, etc.
\end{corollary}


Since diagonal matrices are in bijection with integer compositions,
the number of diagonal matrices in $\Par_n$ is $2^{n-1}$.  Although
the bidiagonal matrices do not admit as compact a description in terms
of the corresponding inversion tables, we can still count them using
the so-called transfer-matrix method~\cite[\S 4.7]{stanleyI}. Consider the matrix
$$B = \Mat{
  \{1,2\}   &  \{3\}    & \emptyset & \emptyset \\
  \emptyset & \emptyset & \{5\}     & \emptyset \\
  \emptyset & \emptyset & \{4,6\}   & \emptyset \\
  \emptyset & \emptyset & \emptyset & \{7\} \\  
}.
$$ More specifically consider creating $B$ by starting with the empty
matrix, $\epsilon$, and inserting the elements $1$, \dots, $7$ one at
a time:
\begin{align*}
  \epsilon &\to
  \mat{\{1\}} \\[-2ex] &\to
  \mat{\{1,2\}} \to
  \begin{aligned}[t]
    \mat{
      \{1,2\}   & \{3\}\\
      \emptyset & \emptyset
    } & \to
    \mat{
      \{1,2\}   & \{3\}     & \emptyset \\
      \emptyset & \emptyset & \emptyset \\
      \emptyset & \emptyset & \{4\} \\
    } \\ & \to
    \mat{
      \{1,2\}   & \{3\}     & \emptyset\\
      \emptyset & \emptyset & \{5\} \\
      \emptyset & \emptyset & \{4\} \\
    } \\[-0.8ex] & \to
    \mat{
      \{1,2\}   & \{3\}     & \emptyset\\
      \emptyset & \emptyset & \{5\} \\
      \emptyset & \emptyset & \{4,6\} \\
    } \to
    \mat{
      \{1,2\}   & \{3\}     & \emptyset &\emptyset\\
      \emptyset & \emptyset & \{5\}     &\emptyset\\
      \emptyset & \emptyset & \{4,6\}   &\emptyset\\
      \emptyset & \emptyset & \emptyset &\{7\}\\
    }
  \end{aligned}
\end{align*}
We shall encode (some aspects of) this process like this:
$$
\epsilon \to
\e \to
\e \to
\t 1, 1, 0, \to
\p 1, 0, 0, 1, \to
\p 1, 0, 1, 1, \to
\p 1, 0, 1, 1, \to
\p 1, 1, 0, 1,.
$$ Here, \!\scalebox{0.9}{\e}\! denotes any $1\times 1$ matrix whose
only entry is a non-empty set; \!\scalebox{0.85}{\t 1, 1, 0,}\!
denotes any $2\times 2$ matrix whose black entries are non-empty;
\!\scalebox{0.8}{\p 1, 0, 0, 1,}\! denotes any matrix of dimension $3$
or more, whose entries in the bottom right corner match the picture,
that is, the black entries are non-empty; etc. The sequence of
pictures does not, in general, uniquely determine a bidiagonal
matrix, but each picture contains enough information to tell what
pictures can possibly follow it. The matrix below gives all possible
transitions (a $q$ records when a new column, and row, is created):
$$
\begin{array}{ccccccccccccccc}
  & \raisebox{-1pt}{$\epsilon$} &
  \e &           
  \t 1, 0, 1, &  
  \t 1, 1, 0, &  
  \t 1, 1, 1, &  
  \p 0, 1, 0, 1, &
  \p 0, 1, 1, 0, &
  \p 0, 1, 1, 1, &
  \p 1, 0, 0, 1, &
  \p 1, 0, 1, 0, &
  \p 1, 0, 1, 1, &
  \p 1, 1, 0, 1, &
  \p 1, 1, 1, 0, &
  \p 1, 1, 1, 1, \\[2ex]
  \epsilon       & 0 & q & 0 & 0 & 0 & 0 & 0 & 0 & 0 & 0 & 0 & 0 & 0 & 0 \\[0.3ex]
  \e             & 0 & 1 & q & q & 0 & 0 & 0 & 0 & 0 & 0 & 0 & 0 & 0 & 0 \\[0.3ex]
  \t 1, 0, 1,    & 0 & 0 & 1 & 0 & 1 & q & q & 0 & 0 & 0 & 0 & 0 & 0 & 0 \\
  \t 1, 1, 0,    & 0 & 0 & 0 & 1 & 1 & 0 & 0 & 0 & q & q & 0 & 0 & 0 & 0 \\
  \t 1, 1, 1,    & 0 & 0 & 0 & 0 & 2 & 0 & 0 & 0 & 0 & 0 & 0 & q & q & 0 \\
  \p 0, 1, 0, 1, & 0 & 0 & 0 & 0 & 0 & q + 1 & q & 1 & 0 & 0 & 0 & 0 & 0 & 0 \\
  \p 0, 1, 1, 0, & 0 & 0 & 0 & 0 & 0 & 0 & 1 & 1 & q & q & 0 & 0 & 0 & 0 \\
  \p 0, 1, 1, 1, & 0 & 0 & 0 & 0 & 0 & 0 & 0 & 2 & 0 & 0 & 0 & q & q & 0 \\
  \p 1, 0, 0, 1, & 0 & 0 & 0 & 0 & 0 & 0 & 0 & 0 & 1 & 0 & 1 & 0 & 0 & 0 \\
  \p 1, 0, 1, 0, & 0 & 0 & 0 & 0 & 0 & 0 & 0 & 0 & q & q + 1 & 1 & 0 & 0 & 0 \\
  \p 1, 0, 1, 1, & 0 & 0 & 0 & 0 & 0 & 0 & 0 & 0 & 0 & 0 & 2 & q & q & 0 \\
  \p 1, 1, 0, 1, & 0 & 0 & 0 & 0 & 0 & q & q & 0 & 0 & 0 & 0 & 1 & 0 & 1 \\
  \p 1, 1, 1, 0, & 0 & 0 & 0 & 0 & 0 & 0 & 0 & 0 & q & q & 0 & 0 & 1 & 1 \\
  \p 1, 1, 1, 1, & 0 & 0 & 0 & 0 & 0 & 0 & 0 & 0 & 0 & 0 & 0 & q & q & 2
\end{array}
$$ We would like to enumerate paths that start with $\epsilon$ and
end in a configuration with no empty rows or columns. Letting $M$
denote the above transfer-matrix, this amounts to calculating the first
coordinate in
$$(1-xM)^{-1}[1\, 1\, 1\, 0\, 1\, 1\, 0\, 1\, 0\, 0\, 1\, 1\, 0\, 1]^T.
$$

\begin{proposition}
  We have
  $$\sum_{n\geq 0}\sum_{A\in\BiPar_n} q^{\dim(A)}x^n =\frac
  {2x^3  - (q + 5)x^2 + (q + 4)x - 1}
  {2(q^2 + q + 1)x^3 - (q^2 + 4q + 5)x^2 + 2(q + 2)x - 1},
  $$
  where $\BiPar_n$ is the collection of bidiagonal matrices in $\Par_n$.
\end{proposition}

We find it interesting that the number of bidiagonal matrices in
$\Par_n$ is given by the sequence~\cite[A164870]{oeis}, which
corresponds to permutations of $[1,n]$ which are sortable by two
pop-stacks in parallel. In terms of pattern avoidance those are the
permutations in the class
$$\sym_n(3214, 2143, 24135, 41352, 14352, 13542, 13524).$$ See
Atkinson and Sack ~\cite{as}. Moreover, there are exactly $2^{n-1}$
permutations of $[1,n]$ which are sortable by one pop-stack; hence
equinumerous with the diagonal partition matrices. One might then
wonder about permutations which are sortable by three pop-stacks in
parallel. Are they equinumerous with tridiagonal partition matrices?
Computations show that this is not the case: For $n=6$ there are $646$
tridiagonal partition matrices, but only $644$ permutations which are
sortable by three pop-stacks in parallel. For more on the enumeration of 
permutations sortable by pop stacks in parallel see Smith and 
Vatter~\cite{sv}.

\section{Composition matrices and $(2+2)$-free posets}
\label{lpsec}

Consider Definition~\ref{def:M}. Define a \emph{composition matrix} to be a
matrix that satisfies conditions (i) and (ii), but not necessarily
(iii). Let $\Comp_n\supseteq\Par_n$ denote the set of all composition
matrices on $[1,n]$. The smallest example of a composition matrix that
is not a partition matrix is
$$\Mat{\{2\} &\emptyset\\ \emptyset & \{1\}}.
$$

In this section we shall give a bijection from $\Comp_n$ to the set of
$(2+2)$-free posets on $[1,n]$. This bijection will factor through a
certain union of Cartesian products that we now define. Given a set
$X$, let us write ${ X \choose x_1,\dots , x_{\ell}}$ for the
collection of all sequences $(X_1,\dots ,X_{\ell})$ that are ordered
set partitions of $X$ and $|X_i|=x_i$ for all $i\in [1,\ell]$.  For a
sequence $(a_1,\dots, a_i)$ of numbers let
$$\asc(a_1,\dots,a_i)=|\{j\in[1,i-1]:a_j<a_{j+1}\}|.
$$ Following Bousquet-M\'elou et al.\ ~\cite{bcdk} we define a
sequence of non-negative integers $\alpha=(a_1,\dots , a_n)$ to be an
\emph{ascent sequence} if $a_1=0$ and $a_{i+1} \in [0,1+\asc(a_1,\dots
  , a_i)]$ for $0<i<n$.  Let $\Ascent_n$ be the collection of ascent
sequences of length $n$. Define the {\it run-length record} of
$\alpha$ to be the sequence that records the multiplicities of
adjacent values in $\alpha$. We denote it by $\RLR(\alpha)$. In other
words, $\RLR(\alpha)$ is the sequence of second coordinates in
$\RLE(\alpha)$, the run-length encoding of $\alpha$. For instance, 
$$\RLR(0,0,0,0,1,1,0,2,3,3)=(4,2,1,1,2).
$$

Finally we are in a position to define the set through which our bijection
from $\Comp_n$ to $(2+2)$-free posets on $[1,n]$ will factor. Let
$$ 
\Mad_n = \bigcup_{\alpha\in \Ascent_n}\{ \alpha \}\times 
    { [1,n] \choose \RLR(\alpha) }.
$$

Let $\robbieclass_n$ be the collection of upper triangular matrices
that contain non-negative integers whose entries sum to $n$ and such
that there is no column or row of all zeros. Dukes and
Parviainen~\cite{dp} presented a bijection
$$\Gamma: \robbieclass_n\to \Ascent_n.
$$

Given $A \in \robbieclass_n$, let $\nz(A)$ be the number of non-zero
entries in $A$.  Since it follows from \cite[Thm.\ 4]{dp} that $A$ may be uniquely
constructed, in a step-wise
fashion, from the ascent sequence $\Gamma(A)$, we may associate to
each non-zero entry $A_{ij}$ its time of creation $T_{A}(i,j) \in
[1,\nz(A)]$. By defining $T_{A}(i,j)=0$ if $A_{ij}=0$ we may view
$T_{A}$ as a $\dim(A)\times \dim(A)$ matrix. Define
$\Seq(A)=(y_1,\dots , y_{\nz(A)})$ where $y_t=A_{ij}$ and
$T_{A}(i,j)=t$.  
\begin{example}
  With
  \begin{align*}
    A &=\Mat{ 
      3 & 0 & 3 & 1 \\
      0 & 1 & 1 & 1 \\
      0 & 0 & 0 & 1 \\
      0 & 0 & 0 & 2
    }
    \shortintertext{ we have }
    T_{A} &=\Mat{ 
      1 & 0 & 5 & 8 \\
      0 & 2 & 4 & 7 \\
      0 & 0 & 0 & 6 \\
      0 & 0 & 0 & 3
    }
    \shortintertext{and}
    \Seq(A) &= (3,1,2,1,3,1,1,1).
  \end{align*}
\end{example}

\begin{lemma}\label{olem}
  Given $A \in \robbieclass_n$, we have that $\Seq(A) =
  \RLR(\Gamma(A))$.
\end{lemma}

\begin{proof}
  This is a straightforward consequence of the construction rules
  given by Dukes and Parviainen~\cite{dp}.
\end{proof}

For a matrix $A \in \Comp_n$ define $\Card(A)$ as the matrix obtained
from $A$ by taking the cardinality of each of its entries. Note that
$A\mapsto\Card(A)$ is a surjection from $\Par_n$ onto $\robbieclass_n$.
Define $\setsystem(A)$ as the ordered set partition $(X_1,\dots ,
X_{\nz(A)})$, where $X_t=A_{ij}$ for $t=T_{\Card(A)}(i,j)$. Finally,
define $\jumpmap:\Comp_n\to\Mad_n$ by
$$\jumpmap(A) = \bigl(\,\Gamma(\Card(A)),\,\setsystem(A)\,\bigr).
$$

\begin{example}
  Let us calculate $\jumpmap(A)$ for\smallskip
  $$
  A = \Mat{
    \{3,8\}   & \{6\}     & \emptyset \\
    \emptyset & \{2,5,7\} & \emptyset \\
    \emptyset & \emptyset & \{1,4\}
  }.
  $$
  We have
  $$
  \Card(A) = \Mat{ 
    2 & 1 & 0 \\
    0 & 3 & 0 \\
    0 & 0 & 2
  };\quad
  T_{\Card(A)} = \Mat{ 
    1 & 3 & 0 \\
    0 & 2 & 0 \\
    0 & 0 & 4
  }
  $$
  and
  \begin{align*}
    \jumpmap(A)
    &= \big(\,\Gamma(\Card(A)),\, \setsystem(A)\,\big) \\
    &= \big(\,(0,0,1,1,1,0,2,2),\, \{3,8\}\{2,5,7\}\{6\}\{1,4\}\,\big).
  \end{align*}
\end{example}

We now define a map $\duckmap:\Mad_n\to\Comp_n$. For
$(w,\chi)\in\Mad_n$ with $\chi=(X_1,\dots,X_k)$ let $\duckmap(w,\chi)
= A$, where $A_{ij}=X_t$, $t=T_B(i,j)$ and $B=\Gamma^{-1}(w)$. It is
easy to verify that $\jumpmap(\Comp_n)\subseteq\Mad_n$,
$\duckmap(\Mad_n)\subseteq\Comp_n$,
$\duckmap(\jumpmap(w,\chi))=(w,\chi)$ for $(w,\chi)\in\Mad_n$, and
$\jumpmap(\duckmap(A))=A$ for $A\in\Comp_n$. Thus we have the following
theorem.

\begin{theorem}
  The map $\jumpmap: \Comp_n \to \Mad_n$ is a bijection and
  $\duckmap$ is its inverse.
\end{theorem}

Next we will give a bijection $\madmap$ from $\Mad_n$ to $\Poset_n$, the set of
$(2+2)$-free posets on $[1, n]$.
Recall that a poset $P$ is $(2+2)$-free if it does not contain
an induced subposet that is isomorphic to $2+2$, the union of two disjoint
$2$-element chains.
Let
$(\alpha,\chi)\in \Mad_n$ with $\chi=(X_1,\dots ,X_{\ell})$.  Assuming that
$X_i=\{x_1,\dots, x_k\}_<$ define the word $\hat{X_i}=x_1\dots x_k$
and let $\hat{\chi}= \hat{X_1}\dots\hat{X_{\ell}}$. From this,
$\hat{\chi}$ will be a permutation of the elements $[1, n]$.  Let
$\hat{\chi}(i)$ be the $i$th letter of this permutation.

For $(\alpha,\chi)\in \Mad_n$ define $\madmap(\alpha,\chi)$ as follows:
Construct the poset element by element according to the construction
rules of \cite{bcdk} on the ascent sequence $\alpha$.  Label with
$\hat{\chi}(i)$ the element inserted at step $i$.

The inverse of this map is also straightforward to state and relies on
the following crucial observation ~\cite[Prop. 3]{dkrs} concerning
indistinguishable elements in an unlabeled $(2+2)$-free poset.  Two
elements in a poset are called \emph{indistinguishable} if they obey the same
relations relative to all other elements.

Let $P$ be an unlabeled poset that is constructed from the ascent
sequence $\alpha=(a_1,\dots , a_n)$. Let $p_i$ and $p_j$ be the
elements that were created during the $i$th and $j$th steps of the
construction given in \cite[Sect. 3]{bcdk}. The elements $p_i$ and
$p_j$ are indistinguishable in $P$ if and only if
$a_{i}=a_{i+1}=\dots =a_j$.

Define $\pirate: \Poset_n \to \Mad_n $ as follows: Given $P \in
\Poset_n$ let $\pirate(P)=(\alpha,\chi)$ where $\alpha$ is the ascent sequence
that corresponds to the poset $P$ with its labels removed, and $\chi$
is the sequence of sets $(X_1,\dots,X_m)$ where $X_i$ is the set of
labels that corresponds to all the indistinguishable elements of $P$
that were added during the $i$th run of identical elements in the
ascent sequence.

\begin{example}
  Consider the $(2+2)$-free poset
  $$ P = \;
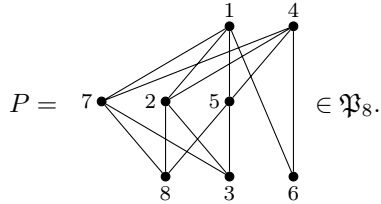
  \begin{tikzpicture}[xscale=0.85, baseline=53pt]
    \tikzstyle{every node} = [font=\footnotesize];
    \tikzstyle{disc} = [ 
      circle,fill=black,draw=black, 
      minimum size=3.3pt, inner sep=0pt];
    \path
    node [disc] (1) at (3, 3) {}
    node [disc] (2) at (2, 2) {}
    node [disc] (3) at (3, 1) {}
    node [disc] (4) at (4, 3) {}
    node [disc] (5) at (3, 2) {}
    node [disc] (6) at (4, 1) {} 
    node [disc] (7) at (1, 2) {}
    node [disc] (8) at (2, 1) {};
    \draw 
    (1) node[above] {1} -- (2) node[left]  {2}
    (1)                 -- (5) node[left]  {5}
    (1)                 -- (6) node[below] {6}
    (1)                 -- (7) node[left]  {7}
    (2)                 -- (3) node[below] {3}
    (2)                 -- (8) node[below] {8}
    (4) node[above] {4} -- (2)
    (4)                 -- (5)
    (4)                 -- (6)
    (4)                 -- (7)
    (5)                 -- (3)
    (5)                 -- (8)
    (7)                 -- (3)
    (7)                 -- (8)
    ;
  \end{tikzpicture}\in\Poset_8.
  $$ The unlabeled poset corresponding to $P$ has ascent sequence
  $(0,0,1,1,1,0,2,2)$. There are four runs in this ascent
  sequence. The first run of two 0s inserts the elements $3$ and $8$, so we
  have $X_1=\{3,8\}$.  Next the run of three 1s inserts elements $2$, $5$
  and $7$, so $X_2=\{2,5,7\}$. The next run is a run containing a single
  $0$, and the element inserted is $6$, so $X_3=\{6\}$. The final run
  of two 2s inserts elements $1$ and $4$, so $X_4=\{1,4\}$. Thus we
  have
  $$\pirate(P) = \big(\,(0,0,1,1,1,0,2,2),\, \{3,8\}\{2,5,7\}\{6\}\{1,4\}\,\big).
  $$
\end{example}

It is straightforward to check that $\madmap$ and $\pirate$ are each
others inverses. Consequently, we have the following theorem.

\begin{theorem}\label{thm:madmap}
  The map $\madmap: \Mad_n \to \Poset_n$ is a bijection and $\pirate$ 
  is its inverse.
\end{theorem}

Let $\C_n$ be the collection of composition matrices $M$ on $[1,n]$
with the following property: in every row of $M$, the entries are
increasing from left to right. An example of a matrix in $\C_8$ is
$$M'=\Mat{ 
  \{4,6\}   & \{8\}     & \emptyset \\
  \emptyset & \{1\}     & \{7\} \\
  \emptyset & \emptyset & \{2,3,5\} 
}.
$$
Every matrix $M \in \C_n$ may be written as a unique pair $(w(M),\hat{\chi}(M))$ where:
\begin{itemize}
\item $w(M)$ is the non-decreasing inversion table that via Theorem~8
  corresponds to $\mono(M)$, the matrix in $\Mono_n$ that is formed
  from $M$ the following way: replace the entries in $M$ from left to
  right, beginning with the first row, with the values $1,\ldots,n$,
  in that order. For the example above we have
  $$\mono(M')=\Mat{ \{1,2\} & \{3\} & \emptyset \\
	\emptyset & \{4\} & \{5\} \\
	\emptyset & \emptyset & \{6,7,8\} } \in \Mono_8
  $$
  and $w(M')=(0,0,0,2,2,4,4,4)$.
\item $\chi(M)=(X_1,\ldots ,X_k)$ where $X_i$ is the union of the
  entries in row $i$ of $M$, and $\hat{\chi}(M)$ is the permutation of
  $[n]$ achieved by removing the parentheses from $\chi(M)$; see
  paragraph after Theorem 16.  For the above example,
  $\chi(M')=(\{4,6,8\},\{1,7\},\{2,3,5\})$ and
  $\hat{\chi}(M')=(4,6,8,1,7,2,3,5)$.
\end{itemize}

Given $M \in \C_n$ with $w(M)=(w_1,\ldots,w_n)$ and
$\hat{\chi}(M)=(\hat{\chi}_1,\ldots , \hat{\chi}_n)$, let us define
the sequence $\pfa(M)=(a_1,\ldots ,a_n)$ by $a_i=1+w_j$ where
$i=\hat{\chi}_j$.  
(For the small example above, we have $\pfa(M')=(3,5,5,1,5,1,3,1)$.)
Recall \cite[p. 94]{stanleyII} that a sequence
$\pfa=(a_1,\ldots,a_n)\in [n]^n$ is a {\it{parking function}} if and
only if the increasing rearrangement $b_1\leq b_2 \leq \cdots \leq
b_n$ of $a_1,\ldots,a_n$ satisfies $b_i\leq i$.

\begin{theorem}
  Matrices in $\C_n$ are in one-to-one correspondence with parking
  functions of order $n$.  The parking function that corresponds to
  the matrix $M \in \C_n$ is $(a_1,\ldots,a_n)$ where
  $a_{\hat{\chi}_j} = 1+w_j$ and $M \leftrightarrow
  (w(M),\hat{\chi}(M))$.
\end{theorem}

\subsection{Diagonal and bidiagonal composition matrices}

Consider the set of diagonal matrices $\DiComp_n$ in $\Comp_n$.  Under
the bijection $\jumpmap$ these matrices map to pairs
$(\alpha,\chi)\in\Mad_n$ where $\alpha$ is a non-decreasing ascent
sequence. Applying $\madmap$ to $\jumpmap(\DiComp_n)$ we get the
collection of $(2+2)$-free posets $P$ which have the property that
every element at level $j$ covers every element at level $j-1$, for
all levels $j$ but the first. Reading the levels from bottom to top we
get an ordered set partition of $[1,n]$. It is not hard to see that
this ordered set partition is, in fact, $\chi$. For instance,\vspace{-1ex}
$$
\Mat{
  \{4\}     & \emptyset & \emptyset \\
  \emptyset & \{1,3\}   & \emptyset \\
  \emptyset & \emptyset & \{2,5\}
}
\;\,\leftrightarrow\;\,
\big((0,1,1,2,2), \{4\}\{1,3\}\{2,5\}\big)
\;\,\leftrightarrow\;\,
\begin{tikzpicture}[xscale=0.4, yscale=0.5, baseline=23pt]
  \tikzstyle{every node} = [font=\footnotesize];
  \tikzstyle{disc} = [ 
    circle,fill=black,draw=black, 
    minimum size=3.3pt, inner sep=0pt];
  \path
  node [disc] (1) at (1, 2) {}
  node [disc] (2) at (1, 3) {}
  node [disc] (3) at (3, 2) {}
  node [disc] (4) at (2, 1) {}
  node [disc] (5) at (3, 3) {};
  \draw 
  (2) node[above] {2} -- (1) node[left]  {1}
  (1)                 -- (5) node[above] {5}
  (2)                 -- (3) node[right] {3}
  (4) node[below] {4} -- (1)
  (4)                 -- (3)
  (5)                 -- (3)
  ;
\end{tikzpicture}.
$$ It follows that there are exactly $k!S(n,k)$ diagonal composition
matrices of size $n$ and dimension $k$, where $S(n,k)$ is the number of
partitions of an
$n$ element set into $k$ parts (a \emph{Stirling number of the second
kind}). For bidiagonal composition matrices the situation is a bit more
complicated:

\begin{proposition}
  We have
  $$
  \sum_{n\geq 0}\sum_{A\in\BiComp_n} q^{\dim(A)}\frac{x^n}{n!} = 
  \frac{qe^{2x} - qe^x - 1}
       {(1-q)qe^{2x} + 2q^2e^x - q^2 - q - 1},
  $$
  where $\BiComp_n$ is the collection of bidiagonal matrices in $\Comp_n$.
\end{proposition}

\begin{proof}
  Let $\mathcal{B}_n$ be the collection of binary bidiagonal matrices
  in $\robbieclass_n$. A matrix $A\in\BiComp_n$ can in a natural and
  simple way be identified with a pair $(B,\chi)$ where $B\in
  \mathcal{B}_k$, $k=\nz(\Card(A))$, and $\chi$ is an ordered set
  partition of $[1,n]$:
  $$
  \Mat{
    \{3,8\}   & \{6\}     & \emptyset \\
    \emptyset & \{2,5,7\} & \emptyset \\
    \emptyset & \emptyset & \{1,4\}
  }\;\,\leftrightarrow\;\,
  \left(
  \Mat{
    1 & 1 & 0 \\
    0 & 1 & 0 \\
    0 & 0 & 1
  },
  \{3,8\}\{6\}\{2,5,7\}\{1,4\}
  \right)
  $$ Thus, if $F(q,x)$ is the ordinary generating function for
  matrices in $\mathcal{B}_n$ counted by dimension and size, then
  $F(q,e^x-1)$ is the exponential generating function we require.
  This is because $F(q,x)$ is also the generating
  function for pairs $(B,\pi)$ where $B\in\mathcal{B}_n$ and
  $\pi\in\sym_n$, and $e^x-1$ is the exponential generating function
  for non-empty sets.

  We now derive $F(q,x)$ using the transfer-matrix method. We grow the
  matrices in $\mathcal{B}_n$ from left to right by adding new
  columns, and within a column we add ones from top to bottom:\\[-4ex]
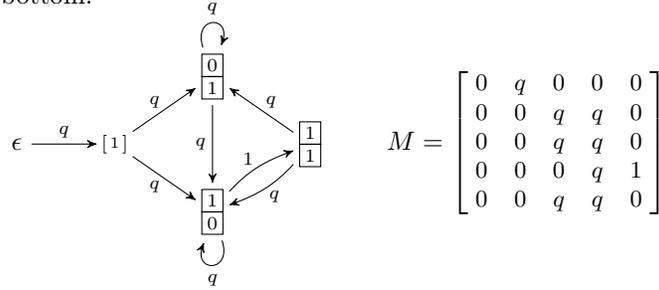
  $$
  \begin{tikzpicture}[baseline=-2pt,>=stealth', yscale=0.9, xscale=1.3]
    \tikzstyle{every node} = [font=\scriptsize];
    \path
    node[font=\normalsize] (e)  at (0, 0) {$\epsilon$}
    node[inner sep=1.5pt]  (1)  at (1, 0) {\one}
    node[inner sep=0pt]    (01) at (2, 1) {\two 0,1,}
    node[inner sep=0pt]    (11) at (3, 0) {\two 1,1,}
    node[inner sep=0pt]    (10) at (2,-1) {\two 1,0,}
    ;
    \path[->] (e)  edge               node[above=-1pt] {$q$}             (1);
    \path[->] (1)  edge               node[pos=0.35, above=-0.1pt] {$q$} (01);
    \path[->] (01) edge[loop above]   node{$q$}                          (01);
    \path[->] (1)  edge               node[pos=0.35, below=-0.1pt] {$q$} (10);
    \path[->] (10) edge[loop below]   node{$q$}                          (10);
    \path[->] (01) edge               node[left=-1pt]{$q$}               (10);
    \path[->] (10) edge[bend left=15] node[pos=0.35, above=-0.1pt] {$1$} (11);
    \path[->] (11) edge               node[pos=0.35, above=-0.1pt] {$q$} (01);
    \path[->] (11) edge[bend left=15] node[pos=0.35, below=-0.1pt] {$q$} (10);
  \end{tikzpicture}
  \qquad\;
  M = \Mat{
     0        & q         & 0         & 0         & 0 \\
     0        & 0         & q         & q         & 0 \\ 
     0        & 0         & q         & q         & 0 \\ 
     0        & 0         & 0         & q         & 1 \\ 
     0        & 0         & q         & q         & 0
  }\vspace{-1ex}
  $$ Here $\epsilon$ is the empty matrix; $\one$ is the $1\times 1$
  identity matrix; \!\raisebox{3pt}{\scalebox{0.8}{\two 0,1,}}\!
  denotes any matrix in $\mathcal{B}_n$ of dimension 2 or more whose
  bottom most entries in the last column are 0 and 1; etc. Calculating
  the first entry in $(1-xM)^{-1}[1\,1\,1\,0\,1]^T$ we find that
  $$
  F(q,x) = \sum_{n\geq 0}\sum_{A\in\mathcal{B}_n} q^{\dim(A)}x^n =
  \frac{qx^2 + qx - 1}{(1 - q)qx^2 + 2qx - 1},
  $$ and on simplifying $F(q,e^x-1)$ we arrive at the claimed
  generating function.
\end{proof}

\section{The number of $(2+2)$-free posets on $[1,n]$}\label{pcp}

Let us consider \emph{plane (2+2)-free posets} on $[1,n]$. That is,
$(2+2)$-free posets on $[1,n]$ with a canonical embedding in the
plane. For instance, these are six \emph{different} plane $(2+2)$-free
posets on $[1,3]=\{1,2,3\}$:
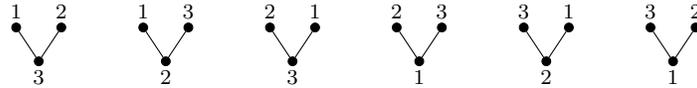
$$
\V{1}{2}{3}\qquad
\V{1}{3}{2}\qquad
\V{2}{1}{3}\qquad
\V{2}{3}{1}\qquad
\V{3}{1}{2}\qquad
\V{3}{2}{1}
$$ By definition, if $u_n$ is the number of unlabeled $(2+2)$-free
posets on $n$ nodes, then $u_nn!$ is the number of plane $(2+2)$-free
posets on $[1,n]$. In other words, we may identify the set of plane
$(2+2)$-free posets on $[1,n]$ with the Cartesian product
$\UP_n\times\sym_n$, where $\UP_{n}$ denotes the set of unlabeled
$(2+2)$-free posets on $n$ nodes and $\sym_n$ denotes the set of
permutations on $[1,n]$. We shall demonstrate the isomorphism
\begin{equation}\label{eq:plane}
  \bigcup_{\pi\in\sym_n}\Poset(\Cyc(\pi)) \,\simeq\, \UP_n\times\sym_n,
\end{equation}
where $\Cyc(\pi)$ is the set of (disjoint) cycles of $\pi$ and
$\Poset(\Cyc(\pi))$ is the set of $(2+2)$-free posets on those cycles.
As an illustration we consider the case $n=3$. On the right-hand side
we have $|\UP_3\times\sym_3|=|\UP_3||\sym_3| = 5\cdot 6=30$ plane
$(2+2)$-free posets. Taking the cardinality of the left-hand side we get
\begin{gather*}
  |\Poset\{\raisebox{0.1ex}{$\scriptstyle (1),(2),(3)$}\}| +
  |\Poset\{\raisebox{0.1ex}{$\scriptstyle (1), (2 3)$}\}| +
  |\Poset\{\raisebox{0.1ex}{$\scriptstyle (1 2), (3)$}\}| +
  |\Poset\{\raisebox{0.1ex}{$\scriptstyle (2), (1 3)$}\}| +
  |\Poset\{\raisebox{0.1ex}{$\scriptstyle (1 2 3)$}\}| +
  |\Poset\{\raisebox{0.1ex}{$\scriptstyle (1 3 2)$}\}| \\
  = |\Poset_3| +  3|\Poset_2| + 2|\Poset_1|
  = 19 + 3\cdot 3 + 2\cdot 1 = 30.
\end{gather*}

Bousquet-M\'elou et al.\ ~\cite{bcdk} gave a bijection $\Psi$ from
$\UP_n$ to $\Ascent_n$, the set of ascent sequences of length
$n$. Recall also that in Theorem~\ref{thm:madmap} we gave a bijection
$\phi$ from $\Poset_n$ to $\Mad_n$. Of course, there is nothing
special about the ground set being $[1,n]$ in
Theorem~\ref{thm:madmap}; so, for any finite set $X$, the map
$\madmap$ can be seen as a bijection from $(2+2)$-free posets on $X$
to the set
$$ 
\Mad(X) = \bigcup_{\alpha\in \Ascent_{|X|}}\{ \alpha \}\times { X \choose \RLR(\alpha) }.
$$ 
In addition, the fundamental transformation~\cite{cf} is a
bijection between permutations with exactly $k$ cycles and
permutations with exactly $k$ left-to-right minima. Putting these
observations together it is clear that to show \eqref{eq:plane} it
suffices to show
\begin{equation}\label{eq:plane-var}
  \bigcup_{\pi\in\sym_n}\Mad(\LMin(\pi)) \,\simeq\, \Ascent_n\times\sym_n,
\end{equation}
where $\LMin(\pi)$ is the set of segments obtained by breaking $\pi$
apart at each left-to-right minima. For instance, the left-to-right minima of 
$\pi=5731462$ are $5$, $3$ and $1$; so $\LMin(\pi)=\{57,3,1462\}$. 

Let us now prove \eqref{eq:plane-var} by giving a bijection $h$ from
the left-hand side to the right-hand side. To this end, fix a
permutation $\pi\in\sym_n$ and let $k=|\LMin(\pi)|$ be the number of
left-to-right minima in $\pi$. Assume that $\alpha=(a_1,\dots,a_k)$ is
an ascent sequence in $\Ascent_k$ and that $\chi=(X_1,\dots,X_r)$ is
an ordered set partition in ${\LMin(\pi) \choose \RLR(\alpha)}$. To
specify the bijection $h$ let
$$h(\alpha, \chi) = (\beta,\tau)
$$ 
where $\beta\in\Ascent_n$ and $\tau\in\sym_n$ are defined in the
next paragraph.

For each $i\in [1,r]$, first order the blocks of $X_i$ decreasingly
with respect to first (and thus minimal) element, then concatenate the
blocks to form a word $\hat{X}_i$. Define the permutation $\tau$ as
the concatenation of the words $\hat{X}_i$:
$$\tau=\hat{X}_1\dots \hat{X}_k.
$$ Let $i_1=1$, $i_2=i_1+|X_1|$, $i_3=i_2+|X_2|$, etc. By definition,  
these are the indices where the ascent sequence $\alpha$
changes in value. Define $\beta$ by
$$\RLE(\beta) = (a_{i_1},x_1)\dots (a_{i_k},x_k),\text{ where }x_i=|\hat{X}_i|. 
$$

Consider the permutation $\pi=\mathrm{A9B68D4F32C175E}\in\sym_{15}$
(in hexadecimal notation). Then $\LMin(\pi)=\{\mathrm{A, 9B, 68D, 4F, 3,
  2C, 175E}\}$. Assume that 
\begin{align*}
  \alpha &= (0,0,1,2,2,2,0);\\
  \chi   &= \{\mathrm{2C, 68D}\}\{\mathrm{9B}\}\{\mathrm{3, 175E, 4F}\}\{\mathrm{A}\}.
\end{align*}
Then we have $\hat{X}_1=\mathrm{68D2C}$, $\hat{X}_2=\mathrm{9B}$,
$\hat{X}_3=\mathrm{4F3175E}$ and $\hat{X}_4=\mathrm{A}$. Also, $i_1=1$,
$i_2=1+2=3$, $i_3=3+1=4$ and $i_4=4+3=7$. Consequently,
$$
\begin{tikzpicture}[every node/.style={anchor=base}]
  \matrix[matrix of math nodes,
    column sep=0pt, row sep=1.4ex, inner sep=0.1pt, minimum size=1.3ex] {
    \beta\! &\;&\node[anchor=base west]{=};& 
    \,(\hspace{0.5pt}&
    0 &,& 0 &,& 0 &,& 0 &,& 0 &,& 1 &,& 1 &,& 2 &,& 
    2 &,& 2 &,& 2 &,& 2 &,& 2 &,& 2 &,& 0 &\hspace{-0.6pt});\\
    \tau\!  &\;&\node[anchor=base west]{=};&  
    & 6 && 8 && \mathrm{D} && 2 && \mathrm{C} && 9 && \mathrm{B} && 4 
    && \mathrm{F} && 3 && 1 && 7 && 5 && \mathrm{E} && \mathrm{A} &.\\
  };
\end{tikzpicture}
$$

It is clear how to reverse this procedure: Split $\tau$ into segments
according to where $\beta$ changes in value when reading from left to
right. With $\tau$ as above we get 
$$(68\mathrm{D}2\mathrm{C}, 9\mathrm{B}, 4\mathrm{F}3175\mathrm{E},\mathrm{A})
=(\hat{X}_1, \hat{X}_2, \hat{X}_3, \hat{X}_4)
$$ We have thus recovered $\hat{X}_1$, $\hat{X}_2$, etc. Now
$X_i=\LMin(\hat{X}_i)$, and we thus know $\chi$. It only remains to recover $\alpha$.
Assume that $\RLE(\beta)=(b_1,x_1)\dots (b_k,x_k)$, then $\RLE(\alpha) =
(b_1,|X_1|)\dots(b_k,|X_k|)$. This concludes the proof
of~\eqref{eq:plane-var}. Let us record this result.

\begin{theorem} \label{th:gen}
  The map $h:\cup_{\pi\in\sym_n}\Mad(\LMin(\pi))\to\Ascent_n\times\sym_n$ is a bijection.
\end{theorem}

As previously explained, \eqref{eq:plane} also follows from this proposition.
Let us now use \eqref{eq:plane} to derive an exponential generating
function $L(t)$ for the number of $(2+2)$-free posets on
$[1,n]$. Bousquet-M\'elou et al.\ ~\cite{bcdk} gave the following \emph{ordinary}
generating function for \emph{unlabeled} $(2+2)$-free posets on $n$
nodes:
\begin{align*}
P(t) 
&=\sum_{n\ge 0} \ \prod_{i=1}^n \left( 1-(1-t)^i\right) \\
&= 1+ t+ 2t^2+ 5t^3+ 15t^4+ 53t^5+ 217t^6+ 1014t^7+ 5335t^8+ O(t^{9}).  
\end{align*}
This is, of course, also the exponential generating function for plane
$(2+2)$-free posets on $[1,n]$. Moreover, the exponential generating
function for cyclic permutations is $\log(1/(1-t))$. On taking the
union over $n\geq 0$ of both sides of \eqref{eq:plane} it follows
that $L(\log(1/(1-t))) = P(t)$; so $L(t) = P(1-e^{-t})$.

\begin{corollary}
  The exponential generating function for $(2+2)$-free posets is
  \begin{align*}
    L(t)
    &=\sum_{n\ge 0} \ \prod_{i=1}^n \left( 1-e^{-ti}\right) \\
    &= 1+ t+ 3\frac{t^2}{2!}+ 19\frac{t^3}{3!}+ 207\frac{t^4}{4!}
    + 3451\frac{t^5}{5!}+ 81663\frac{t^6}{6!}+ 2602699\frac{t^7}{7!}+ O(t^8).
\end{align*}
\end{corollary}

This last result also follows from a result of Zagier~\cite[Eq.\ 24]{z} and a
bijection, due to Bousquet-M\'elou et\ al.~\cite{bcdk}, between
unlabeled $(2+2)$-free posets and certain matchings.  See also
Exercises 14 and 15 in Chapter 3 of the second edition of Enumerative
Combinatorics volume 1 (available on R. Stanley's homepage).

\end{document}